\documentclass[12pt]{article}
\usepackage{amsthm}
\usepackage{amsmath}
\usepackage{amsfonts}
\usepackage{graphicx}
\usepackage{latexsym}
\usepackage{amssymb}
\usepackage{lmodern}

\newcommand{\abs}[1]{\left|#1\right|}

\newcommand{\field}[1]{\mathbb{#1}}

\newcommand{\Z}{\field{Z}}

\newcommand{\F}{\field{F}}

\newcommand{\cA}{{\cal A}}
\newcommand{\cB}{{\cal B}}
\newcommand{\cC}{{\cal C}}

\newcommand{\cS}{{\cal S}}

\newcommand{\cT}{{\cal T}}

\newcommand{\cR}{{\cal R}}

\newcommand{\cN}{{\cal N}}

\newcommand{\cG}{{\cal G}}

\newcommand{\mmod}{{\textup{mod}}}

\newtheorem{definition}{Definition}
\newtheorem{construction}{Construction}
\newtheorem{theorem}{Theorem}
\newtheorem{lemma}{Lemma}
\newtheorem{remark}{Remark}
\newtheorem{corollary}{Corollary}
\newtheorem{conjecture}{Conjecture}
\newtheorem{example}{Example}

\begin{document}

\bibliographystyle{plain}

\title{
\begin{center}
Steiner Systems over Mixed Alphabet \\
and Related Designs
\end{center}
}
\author{
{\sc Tuvi Etzion}\thanks{Department of Computer Science, Technion,
Haifa 3200003, Israel, e-mail: {\tt etzion@cs.technion.ac.il}.}}

\maketitle

\begin{abstract}
A mixed Steiner system MS$(t,k,Q)$ is a set (code) $\cC$ of words of weight $k$ over an alphabet~$Q$, where not all coordinates of a word have the
same alphabet size, each word of weight~$t$, over~$Q$, has distance $k-t$ from exactly one codeword of $\cC$,
and the minimum distance of the code $2(k-t)+1$.
Mixed Steiner systems are constructed from perfect mixed codes, resolvable designs, large set,
orthogonal arrays, and a new type of pairs-triples design. Necessary conditions for the existence of mixed Steiner systems are presented and
it is proved that there are no large sets of these Steiner systems.
\end{abstract}

\vspace{0.5cm}



\newpage
\section{Introduction}
\label{sec:introduction}

A \emph{Steiner system of order $n$}, S$(t,k,n)$, is a pair $(\cN,B)$, where $\cN$ is an $n$-set
(whose elements are called \emph{points}) and $B$ is a collection
of $k$-subsets (called \emph{blocks}) of~$\cN$, such that each $t$-subset of~$\cN$ is contained
in exactly one block of $B$. A Steiner system can be represented by a binary code~$\cC$ whose codewords have length $n$ and
weight $k$. For each word $x$ of length $n$ and weight $t$, there is exactly one codeword $c \in \cC$ for
which $d(x,c)=k-t$, where $d(y,z)$ is the Hamming distance between the words $y$ and $z$.
As a code an S$(t,k,n)$ has minimum distance $2(k-t)+2$. By abuse of notation we are going to use a mixed notation of sets and words
and it will not be mentioned which one is used.
Steiner systems were extensively studied in the literature for existence and applications.

There are several generalization of Steiner systems, such as group divisible designs~\cite{Han75}, orthogonal arrays~\cite{HSS99},
generalized Steiner systems~\cite{Etz97}, etc.~. The current work considers another generalization of Steiner systems, namely mixed Steiner systems,
which are defined first in~\cite{Etz22}.

A \emph{mixed Steiner system} MS$(t,k,Q)$, over the mixed alphabet $Q=\Z_{q_1} \times \Z_{q_2} \times \cdots \times \Z_{q_n}$,
is a pair $(Q,\cC)$, where $\cC$ is a set of codewords (blocks) of weight $k$, over $Q$, for each word $x$ of weight $t$ over $Q$, there
exists exactly one codeword $c \in \cC$, such that $c$ covers $x$, i.e., $d(x,c)=k-t$, and
the minimum Hamming distance of the code $\cC$ is $2(k-t)+1$. In most cases, we assume that
$q_i \leq q_j$ for $1 \leq i < j \leq n$ and at least two of the $q_i$'s are different, but in
some cases there is no given such order. If all the $q_i$'s are equal, then the design is
a Steiner system when all the $q_i$'s equal 2 and a generalized Steiner system when all the $q_i$'s are equal and greater than 2.

Our goal is to have such a system where the minimum distance is $2(k-t)+1$ as one of the requirements.
But, for some parameters this minimum distance is not possible. When the minimum distance will be different,
i.e., smaller, from $2(k-t)+1$ it will be mentioned and the system will be denoted by MS$_d(t,k,Q)$, where $d$ is
the minimum Hammimg distance of the system (code).
Mixed Steiner systems were defined in~\cite{Etz22} in connection of the existence of perfect codes over a mixed alphabet,
but there was no comprehensive discussion on properties, bounds, and constructions on these systems. This discussion is the
goal of the current work.

The rest of the paper is organized as follows. Section~\ref{sec:conditions} examines some of the necessary conditions
for the existence of mixed Steiner systems. Some of these conditions are generalizations of the conditions for the existence of Steiner
systems, but they are much more evolved. Some of these necessary conditions are shown to be sufficient in this section.
In Section~\ref{sec:trivial} we examine the existence of some trivial
mixed Steiner systems and show that their construction is not always straightforward.
Section~\ref{sec:perfect} examines the connections between mixed Steiner systems and perfect codes
over a mixed alphabet. Many families of mixed Stiener systems are constructed from these mixed perfect codes and necessary conditions
for the existence of these perfect codes are derived from mixed Steiner systems and as a consequence some possible parameters for perfect codes
are eliminated. Large sets of Steiner systems and resolvable Steiner systems can be used for constructions of
mixed Steiner systems.
The same can be done for large sets of mixed Steiner systems, but it is proved that unfortunately
these large sets cannot exist
if there is more than one alphabet size. Section~\ref{sec:pairwise} is devoted to pairwise balanced mixed systems, i.e., mixed
Steiner systems MS$(2,k,Q)$. We show a very simple and efficient construction of such systems based on orthogonal arrays which are
equivalent to a set of mutually orthogonal Latin squares. Some of the codes obtained from orthogonal arrays are MS$_d(2,k,Q)$,
where $d < 2k-1$. In addition, a new combinatorial design of pairwise disjoint one-factors and a set of triples is
used for one of the constructions. 
Conclusion and future research are discussed in Section~\ref{sec:conclusion}.
Our exposition will be on all parameters of mixed Steiner systems, but special emphasis will be given
to mixed Steiner Systems MS$(2,3,\Z_2^n \times \Z_q)$ for $q > 2$.

\section{Necessary Conditions}
\label{sec:conditions}

Before starting our exposition we consider the representation of words (blocks) for a mixed Steiner system.
A~block of size $k$ in a Steiner system S$(t,k,n)$
is represented by a subset of size $k$. A~codeword in a mixed Steiner system MS$(t,k,Q)$ over
$Q=\Z_{q_1} \times \Z_{q_2} \times \cdots \times \Z_{q_n}$ can be represented by a word of length $n$ and weight $k$, where the $i$-th coordinate
contains a letter from the alphabet~$\Z_{q_i}$. The codeword can be represented also as a subset of size $k$, where if a nonzero letter $\alpha$
is in the $i$-th coordinate of the codeword, then the subset contains the element $(i,\alpha)$, and as in a Steiner system the entries with
zeros are not represented. The different representations
will be used and the used one will be understood from the context. The coordinates in a codeword of S$(t,k,n)$ will be cometimes
$1,2,\ldots,n$ and sometimes $0,1,\ldots,n-1$, i.e., the elements of $Z_n$.
Note, that while $\Z_n = \{ 0,1,\ldots, n-1 \}$, $\Z_2^n$ consists of all binary words of length $n$.
If the system is over $\Z_2^n \times \Z_q$, the subset $\{ x_1,x_2,\ldots,x_r,(n+1,\alpha)\}$ correspond to the word in which $x_1,x_2,\ldots,x_r$
are $r$ position with \emph{ones} in the first $n$ positions and in the $(n+1)$-th position there is the nonzero symbol $\alpha$ from~$\Z_q$.

Let $\F_q$ denote the finite field with $q$ elements and the \emph{support} of a word $x$ is the set of nonzero positions in $c$.

Similar to Steiner systems, also mixed Steiner systems have necessary conditions for their existence.
The first two lemmas is a straightforward generalization from the known results on Steiner systems.

\begin{lemma}
The number of codewords in a mixed Steiner system \textup{MS}$(t,k,Q)$, where $Q=\Z_{q_1} \times \Z_{q_2} \times \cdots \times \Z_{q_n}$, is
$$
\left( \sum_{\substack{Y \subseteq [n] \\ \abs{Y}=t}} \prod_{j \in Y} (q_j -1) \right) \text{{\Large /}} \binom{k}{t}~.
$$
\end{lemma}
\begin{proof}
Let $S=(Q,\cC)$ be a mixed Steiner system MS$(t,k,Q)$.
Each word of weight $t$ must be covered by exactly one codeword of $\cC$.
For each $t$~coordinates $i_1,i_2,\ldots,i_t$, $i_1 < i_2 < \cdots < i_t$ there are $\prod_{j=1}^t (q_{i_j}-1)$ words
of weight $t$ over $Q$ whose support is $\{ i_1,i_2,\ldots,i_t\}$. Each of these $\prod_{j=1}^t (q_{i_j}-1)$
words must be covered by exactly one codeword of $\cC$.
A codeword $X$ of weight~$k$ covers exactly $\binom{k}{t}$ words of weight $t$
and hence the claim of the lemma follows.
\end{proof}

\begin{lemma}
\label{lem:derived}
If $(Q \times \Z_q,\cC)$ is a mixed Steiner system \textup{MS}$(t,k,Q \times \Z_q)$,
then for each $\alpha \in \Z_q \setminus \{ \bf0 \}$, the set
$$
\cC_\alpha \triangleq \{ c ~:~ (c,(n+1,\alpha)) \in \cC \}
$$
is a mixed Steiner system MS$(t-1,k-1,Q)$.
\end{lemma}

If $(Q \times \Z_q,\cC)$ is a mixed Steiner system \textup{MS}$(t,k,Q \times \Z_q)$ then the system
MS$(t-1,k-1,Q)$ is called \emph{the derived design} or \emph{the derived system}.
Note, that $\Z_q$ is not necessarily the alphabet of the largest size.

\begin{corollary}
\label{cor:necessary_MST}
A necessary condition for the existence of a mixed Steiner system \textup{MS}$(t,k,Q)$, where
$Q=\Z_{q_1} \times \Z_{q_2} \times \cdots \times \Z_{q_n}$, is that for each $i$, $0 \leq i \leq t-1$,
and each subset $X$ of $[n]$ whose size is $n-i$, we have
$$
\sum_{\substack{Y \subseteq X \\ \abs{Y}=t-i}} \prod_{j \in Y} (q_j -1) \equiv 0 \left( \hspace{-0.1cm} \mmod ~ \binom{k-i}{t-i} \right)~.
$$
\end{corollary}

Lemma~\ref{lem:derived} implies that if there exists a mixed Steiner system MS$(t,k,Q)$, then there exists a mixed Steiner system MS$(1,k-t+1,Q')$,
where $Q'$ is obtained from $Q$ by deleting $t-1$ coordinates and therefore we start by considering the necessary conditions for
the existence of such a system.

\begin{theorem}
\label{thm:lastN}
If there exists a mixed Steiner system \textup{MS}$(1,k,Q \times \Z_q)$, where
$Q = \Z_{q_1} \times \cdots \times \Z_{q_{n-1}}$,
then $\sum_{i=1}^{n-1} (q_i -1) - (q -1)(k-1)$ is not negative and divisible by $k$.
\end{theorem}
\begin{proof}
There are $q -1$ nonzero alphabet letters in $\Z_q$. Each one of these alphabet letters must be contained in a separate block
of size $k$ with $k-1$ other nonzero elements and hence the blocks that contain nonzero alphabet letters
from $\Z_q$ contain $(q -1)(k-1)$ nonzero elements not from~$\Z_q$.
Hence, $\sum_{i=1}^{n-1} (q_i -1) \geq (q -1)(k-1)$, i.e., $\sum_{i=1}^{n-1} (q_i -1) - (q -1)(k-1)$ is not negative.
This difference must be divisible by $k$ since each block that does not contain nonzero elements from~$\Z_q$ contains exactly $k$ nonzero distinct elements.
\end{proof}


%


We continue and consider necessary conditions for some basic mixed Steiner systems. We consider the necessary conditions for
the existence of a mixed Steiner system MS$(t,t+1,\Z_2^n \times Z_q)$. We start with $t=1$.

\begin{theorem}
\label{thm:NS_t=1}
A mixed Steiner system \textup{MS}$(1,2,\Z_2^n \times \Z_q)$, where $q>2$, exists if and only if $n \geq q-1$ and $n-q+1$ is an even integer.
\end{theorem}
\begin{proof}
Each nonzero element of $\Z_q$ is paired with a different coordinate and hence $n \geq q-1$. There are $q -1$ nonzero elements
in $\Z_q$ and hence the other $n - q+1$ coordinates are partitioned into disjoint pairs and therefore $n-q+1$ must be an even integer.
A construction of such a mixed Steiner system MS$(1,2,\Z_2^n \times Z_q)$ is done exactly following these lines.
The first set of blocks of such system are
$$
\{ \{(i,1),(n+1,\alpha_i)\} ~:~ 1 \leq i \leq q -1 \},
$$
where $\alpha_i$ is the $i$-th nonzero element of $\Z_q$.

The second set of blocks of such system are
$$
\{ \{(q +2i,1),(q+2i+1,1)\} ~:~ 0 \leq i \leq \frac{n-q-1}{2} \}.
$$
\end{proof}

The next steps are for $t=2$ and $t=3$, as follows.
\begin{theorem}
\label{thm:necessry2,3}
If there exists mixed Steiner system \textup{MS}$(2,3,\Z_2^n \times \Z_q)$, then $n \geq q$ and also
either 6~divides $n$ and $q$ is even or $n \equiv q \equiv 2$ or $4~ (\mmod ~ 6)$.
\end{theorem}
\begin{proof}
Since the two possible derived systems of an \textup{MS}$(2,3,\Z_2^n \times \Z_q)$ are S$(1,2,n)$ (which implies that $n$ is even) and
MS$(2,3,\Z_2^{n-1} \times Z_q)$, it follows by Theorem~\ref{thm:NS_t=1} that $n-1 \geq q-1$ and $n-q$ is even.
As a consequence we have that both $q$ and $n$ are even.

The number of pairs in $\Z_2^n \times \Z_q$ is $\frac{n(n-1)}{2} + (q-1)n = \frac{n(2q+n-3)}{2}$ and since each triple contains
three pairs, it follows that 3 divides $\frac{n(2q+n-3)}{2}$. Since $q$ and $n$ are even it implies that if $n \equiv 0~(\mmod~6)$
then $q$ is even, but if $n \equiv 2$ or $4~(\mmod~6)$, then $q \equiv n~(\mmod~6)$.
\end{proof}

\begin{theorem}
\label{thm:necessry3,4}
If there exists mixed Steiner system \textup{MS}$(3,4,\Z_2^n \times \Z_q)$, then $n \geq q+1$ and also
either 6 divide $n-1$ and $q$ is even or $n \equiv 3~ (\mmod ~ 12)$ and $q \equiv 2~(\mmod~ 12)$
or $n \equiv 9~ (\mmod ~ 12)$ and $q \equiv 2~(\mmod~ 6)$.
\end{theorem}
\begin{proof}
Since the two possible derived systems of \textup{MS}$(3,4,\Z_2^n \times \Z_q)$ are S$(2,3,n)$ (which implies that $n \equiv 1$ or $3~(\mmod~6)$) and
MS$(2,3,\Z_2^{n-1} \times \Z_q)$, it follows by Theorem~\ref{thm:necessry2,3} that $n \geq q+1$ and $n \equiv 1$ or $3~(\mmod~6)$.
Moreover, if $n \equiv 1~(\mmod~6)$, then $q$ is even and if $n \equiv 3~(\mmod~6)$, then $q \equiv 2~(\mmod ~6)$.

The number of triples in $\Z_2^n \times \Z_q$ is $\frac{n(n-1)(n-2)}{6} + (q-1)\frac{n(n-1)}{2} = \frac{n(n-1)(3q+n-5)}{6}$ and since each quadruple contains
four triples, it follows that 4 divides $\frac{n(n-1)(3q+n-5)}{6}$. This implies that if $n \equiv 3~(\mmod~12)$
then $q \equiv 2~(\mmod~12)$ and if $n \equiv 9~(\mmod~12)$, then $q \equiv 2~(\mmod~6)$.
\end{proof}

We can continue and find the necessary conditions for $t > 3$ in a similar way.
A generalization for Theorem~\ref{thm:NS_t=1} is given as follows:

\begin{theorem}
\label{thm:NS_t=1gk}
A mixed Steiner system \textup{MS}$(1,k,\Z_2^n \times \Z_q)$ exists if and only if $n \geq (k-1)(q-1)$ and $n - (k-1)(q-1)$ is divisible by $k$.
\end{theorem}
\begin{proof}
Each nonzero element of $\Z_q$ is paired with a disjoint $k-1$ coordinates in the different
codewords and hence $n \geq (k-1)(q-1)$. There are $q -1$ nonzero elements
in $\Z_q$ and hence the other $n - (k-1)(q-1)$ coordinates are partitioned into disjoint $k$-subsets and hence $n - (k-1)(q-1)$ must be
divisible by $k$.
Construction of such a mixed Steiner system MS$(1,k,\Z_2^n \times \Z_q)$ is done exactly following these lines.
The first set of blocks of such system are
$$
\{ \{(i,1),(n+1,\alpha_i)\} ~:~ 1 \leq i \leq q -1 \},
$$
where $\alpha_i$ is the $i$-th nonzero element of $\Z_q$.

The second set of blocks of such system are
$$
\{ \{((k-1)(q-1) +ki+1,1),((k-1)(q-1) +ki+2,1),\ldots, ((k-1)(q-1) +k(i+1),1) \}
$$
$$
:~ 0 \leq i \leq \frac{n - (k-1)(q-1)-k}{k} \}.
$$
\end{proof}

Finally, Theorem~\ref{thm:necessry2,3} is generalized as follow.

\begin{theorem}
\label{thm:necessry2,k}
If there exists a mixed Steiner system \textup{MS}$(2,k,\Z_2^n \times \Z_q)$, then $n \geq (k-2)(q-1) +1$, $k-1$ divides $n$,
$q \equiv 2~(\mmod ~ k-1)$, and $k(k-1)$ divides $n(n+2q-3)$.
\end{theorem}
\begin{proof}
Since the two possible derived systems of an \textup{MS}$(2,k,\Z_2^n \times \Z_q)$ are S$(1,k-1,n)$ (which implies that $k-1$ divides $n$) and
MS$(1,k-1,\Z_2^{n-1} \times \Z_q)$, then by Theorem~\ref{thm:NS_t=1gk} we have that $n \geq (k-2)(q-1) +1$ and $k-1$ divides
$n-1 - (k-2)(q-1)$ which together with the previous condition that $k-1$ divides~$n$ implies that $k-1$ divides $q-2$
(using simple algebraic manipulation), i.e, $q \equiv 2~(\mmod ~ k-1)$.

The number of pairs in $\Z_2^n \times \Z_q$ is $\frac{n(n-1)}{2} + (q-1)n = \frac{n(2q+n-3)}{2}$ and since each $k$-subset contains
$\binom{k}{2}$ pairs, it follows that $\binom{k}{2}$ divides $\frac{n(2q+n-3)}{2}$, i.e.,
$k(k-1)$ divides $n(n+2q-3)$.
\end{proof}

\begin{example}
Theorem~\ref{thm:necessry2,k} implies that the necessary conditions for the existence of \textup{MS}$(2,4,\Z_2^n \times \Z_q)$ implies that
$n \equiv 0,~3$ or $9~(\mmod~12)$ and
\begin{enumerate}
\item If $n \equiv 0 ~(\mmod~12)$ then $q \equiv 2~(\mmod~3)$.

\item If $n \equiv 3 ~(\mmod~12)$ then $q \equiv 2~(\mmod~6)$.

\item If $n \equiv 9 ~(\mmod~12)$ then $q \equiv 5~(\mmod~6)$.
\end{enumerate}
\end{example}

A different type of a necessary condition is the
following lemma that generalizes a similar result on perfect mixed codes~\cite{EtGr93}.

\begin{lemma}
\label{lem:lower_bound_length}
If $\cS$ is an \textup{MS}$(t,k,Q \times \Z_q)$, where the length of the codewords is $n$ and $q > 2$, then
$$
n \geq (k-t) q + 2t -k ~.
$$
\end{lemma}
\begin{proof}
Consider the word $(1^{t-1} 0^{n-t} \alpha)$, where $\alpha \in \Z_q \setminus \{ \bf0 \}$. This word of weight $t$ must be covered by a codeword
of weight $k$ that starts with $t-1$ \emph{ones}, ends with $\alpha$ and another $k-t$ nonzero entries in the middle.
Since the minimum distance of $\cS$ is $2(k-t)+1$, it follows that for two words $(1^{t-1} 0^{n-t} \alpha_1)$ and
$(1^{t-1} 0^{n-t} \alpha_2)$, where $\alpha_1 \neq \alpha_2$, $\alpha_1,\alpha_2 \in \Z_q \setminus \{ \bf0 \}$
which are covered by two codewords $c_1$ and $c_2$, respectively,
the $k-t$ nonzero entries in the middle of the two codewords are in disjoint coordinates.

Since $\alpha$ can be chosen in $q -1$ distinct ways, it follows that $n-t \geq (q -1)(k-t)$, i.e.,
$n\geq (k-t) q + 2t -k$.
\end{proof}

\section{Trivial Systems}
\label{sec:trivial}

It is relatively easy to identify trivial mixed Steiner systems. The first family of such systems consists of systems
that contain all the words in the space as follows.
\begin{lemma}
If $Q$ is a mixed alphabet on $n$ coordinates and $k$ is an integer, where $1 \leq k \leq n$, then there exists
a mixed Steiner system \textup{MS}$(k,k,Q)$.
\end{lemma}

The second family of trivial mixed Steiner systems MS$(t,k,Q)$ is associated with the parameter $t=1$.
The necessary condition on Theorem~\ref{thm:lastN} is also sufficient as follows.
The theorem also generalizes Theorem~\ref{thm:NS_t=1gk}.

\begin{theorem}
\label{lem:lastN}
If $Q = \Z_{q_1} \times \Z_{q_2} \times \cdots \times \Z_{q_{n-1}} \times \Z_{q_n}$, $q_i \leq q_j$ for $i < j$, not all the $q_i$'s
are equal, and $\sum_{i=1}^{n-1} (q_i -1) - (q_n -1)(k-1)$ is not negative and divisible by $k$, then
there exists a mixed Steiner system \textup{MS}$(1,k,Q)$ .
\end{theorem}
\begin{proof}
The proof is by induction on $\sum_{i=1}^n (q_i -1)$. Since $\sum_{i=1}^{n-1} (q_i -1) - (q_n -1)(k-1)$ is not negative and divisible by $k$,
it follows that $\sum_{i=1}^n (q_i -1)$ is divisible by $k$.

The basis is $\sum_{i=1}^n (q_i -1)=k$. Since $\sum_{i=1}^{n-1} (q_i -1) - (q_n -1)(k-1)$
is not negative divisible by~$k$, it follows that each $q_i$, $1 \leq i \leq n$, equals to 2
and the system contains exactly one block that contains all the $q_i$'s (which implies that the mixed Steiner system is
a Steiner system S$(1,k,k)$).

For the induction hypothesis, assume that $Q' = \Z_{q'_1} \times \Z_{q'_2} \times \cdots \times \Z_{q'_{n'-1}} \times \Z_{q'_{n'}}$,
$q'_i \leq q'_j$ for $i < j$, not all the $q_i$'s are equal,
$\sum_{i=1}^{n'} (q'_i -1)=rk$, where $r \geq 1$, and there exists a mixed Steiner system MS$(1,k,Q')$.

For the induction step let $\sum_{i=1}^n (q_i -1)=(r+1)k$. We form a block $\{ (i,1) ~:~ n-k+1 \leq i \leq n\}$ which contains
the first nonzero elements of the $k$ largest alphabet sizes. We replace $\Z_{q_i}$ with~$\Z_{q_i -1}$, $n-k+1 \leq i \leq n$ (it can be done since we
used one nonzero element from each of these alphabets) and a new order the alphabet letters if required (this is required only if
$q_{n-k+1}=q_{n-k}$) and some of these alphabets might disappear (if some of the taken $q_i$'s for the block were equal to 2).
With the new alphabets we can use and finish with the induction hypothesis.
\end{proof}

\begin{corollary}
Assume $Q \times \Z_{q_n} = \Z_{q_1} \times \Z_{q_2} \times \cdots \times \Z_{q_{n-1}} \times \Z_{q_n}$, $q_i \leq q_j$
for $i < j$ and not all the $q_i$'s are equal.
There exists a mixed Steiner system \textup{MS}$(1,k,Q \times \Z_{q_n})$ if and only if
$\sum_{i=1}^{n-1} (q_i -1) - (q_n -1)(k-1)$ is not negative and divisible by $k$.
\end{corollary}

\section{Systems from Mixed Perfect Codes}
\label{sec:perfect}

An \emph{$e$-perfect mixed code} $\cC$ is a code over a mixed alphabet $Q= \Z_{q_1} \times \Z_{q_2} \times \cdots \times \Z_{q_{n-1}} \times \Z_{q_n}$
and covering radius $e$ (minimum distance $2e+1$) is a code whose codewords have length $n$ over the alphabet $Q$. For each word $x \in Q$ there exists
a unique codeword $c \in \cC$ such that $d(x,c) \leq e$. It is well-known that over an alphabet $\Z_2^n$ in such a perfect code, that
contains the all-zero codeword, the codewords
of weight $2e+1$ form a Steiner system S$(e,2e+1,n)$.

Perfect mixed codes were considered first by Sch\"{o}nheim~\cite{Sch70} and later by~\cite{Hed77,HeSc71,HeSc72,Lin75}.
All the codes considered in these papers have covering radius 1. Such codes are associated with group partitions.
Such a group partition is also associated with perfect byte-correcting codes which was considered in~\cite{HoPa72} where it was generalized and
extensively studied in~\cite{Etz98}. Group partitions and perfect byte-correcting codes are also associated with spreads and partial
spreads. These partitions were studied in many papers, e.g.,~\cite{BESSSV08,EJSSS08,ESSSV07,ESSSV09,Hed09a,Hed09b,Kha09,SSSV12a,SSSV12b}.
Nonexistence of such codes were considered in a few papers, e.g.,~\cite{EtGr93,Hed75,PSS06,Wee91}. An extensive survey on perfect mixed codes can
be found in Chapter~7 of~\cite{Etz22}.

\begin{theorem}
\label{PER_ST}
The codewords of weight $2e+1$ of an $e$-perfect code $\cC$ (which contains the all-zero word), over a mixed alphabet
${Q=\Z_{q_1} \times \Z_{q_2} \times \cdots \times \Z_{q_n}}$, $q_i \leq q_j$ for $i < j$, form
a mixed Steiner system \textup{MS}$(e+1,2e+1,Q)$.
\end{theorem}
\begin{proof}
Clearly, the all-zero codeword covers all the words of weight at most~$e$, over $Q$, and no word of weight larger than~$e$.
Hence, since the code~$\cC$ contains the all-zero codeword and its minimum distance is $2e+1$,
it follows that it does not contain any codeword of weight between \emph{one} and $2e$.
Therefore, the words of weight~${e+1}$, over~$Q$, must be covered by codewords from $\cC$ of weight $2e+1$.
Each of these words of weight ${e+1}$ must be covered by exactly one codeword of weight~$2e+1$.
This implies that the codewords of weight $2e+1$ in $\cC$ form a mixed Steiner system MS$(e+1,2e+1,Q)$.
\end{proof}

Now, we can use Corollary~\ref{cor:necessary_MST} to obtain the following consequence.

\begin{corollary}
\label{cor:sum_prod}
If $\cC$ is an $e$-perfect mixed code of length $n$, over ${Q=\Z_{q_1} \times \Z_{q_2} \times \cdots \times \Z_{q_n}}$,
$q_i \leq q_j$ for $i < j$, then for each $t$, $1 \leq t \leq e+1$ and each
subset $X$ of~$[n]$ whose size is $n-e+t-1$, we have
$$
\sum_{\substack{Y \subseteq X \\ \abs{Y}=t}} \prod_{j \in Y} (q_j -1) \equiv 0 \left(\hspace{-0.1cm} \mmod ~ \binom{e+t}{t} \right).
$$
\end{corollary}

Most known $e$-perfect mixed codes are for $e=1$. There is only one known family of 2-perfect mixed codes~\cite{EtGr93}.
The codes in this family are of length $2^n +1$, where $n \geq 4$ is an even integer. This code is over $\Z_2^{2^n} \times Z_{2^{n-1}}$
and by Theorem~\ref{PER_ST} we have
\begin{theorem}
\label{thm:mixedPreparata}
There exists a mixed Steiner system \textup{MS}$(3,5, \Z_2^{2^n} \times Z_{2^{n-1}})$ for each even $n \geq 4$.
Such a mixed Steiner system attains the lower bound of Lemma~\ref{lem:lower_bound_length}.
\end{theorem}

We continue with 1-mixed perfect codes based on partitions of the nonzero element of $\F_{q^m}$ into subsets, each one a subspace
if the zero element is added.

\begin{theorem}
\label{thm:mixed_partitions}
Let $\{\cS_1, \cS_2,\ldots,\cS_n \}$ be a partition of $\F^{q^m} \setminus \{ \bf0 \}$ into $n$ subsets such that
$\cT_i \triangleq \cS_i \cup \{ \bf0 \}$ is a subspace of dimension $k_i$ over $\F_q$. The code defined by
$$
\cC \triangleq \left\{ (c_1,c_2,\ldots,c_n) ~:~ c_i \in \cS_i \cup \{ 0 \},~ \sum_{i=1}^n c_i =0  \right\},
$$
where the sum is performed in $\F_{q^m}$,
is a 1-perfect mixed code over ${\cT_1 \times \cT_2  \times \cdots \times \cT_n}$ which is
isomorphic to $\F_{q^{k_1}} \times \F_{q^{k_2}} \times \dots \times \F_{q^{k_n}}$.
\end{theorem}
\begin{proof}
If $c$ and $c'$ are two codewords of $\cC$, then clearly $c+c'$ is also a codeword in $\cC$.
This implies that $\cC$ is a linear code and its minimum distance is the weight of the codeword
of minimum weight in the code. It is readily verified that there is no codeword of weight one in $\cC$. Therefore,
if $(x_1,x_2,\ldots,x_n) \in {\cT_1 \times \cT_2  \times \cdots \times \cT_n}$ has weight two, then assume w.l.o.g. that
$x_i \neq 0$ and $x_j \neq 0$, for some $1 \leq i < j \leq n$. Since $x_i \in \cT_i$, $x_j \in \cT_j$, $\cT_i  \cap \cT_j = \{ \bf0 \}$, and
$\cT_i$, $\cT_j$ are subspaces of $\F_{q^m}$, it follows that $x_i + x_j \neq 0$ and hence $\cC$ does not have a
codeword of weight two. This implies that the minimum distance of $\cC$ is three.

To complete the proof of the claim of the theorem, it suffices to show that for each word
$x=(x_1,x_2,\dots,x_n) \in \cT_1 \times \cT_2 \times \dots \times \cT_n$, there exists
exactly one codeword $c=(c_1,c_2,\ldots,c_n)$ in $\cC$ such that $d(x,c) \leq 1$. Let $s = \sum_{i=1}^n x_i$. If $s=0$,
then, clearly, $x$ is a codeword.  If $s = \alpha \neq 0$, then
let $j$ be the unique integer such that $\alpha \in \cS_j$.
Clearly, since $\cS_j \cup \{ \bf0 \}$ is a subspace and also $x_i \in \cT_i$, for each $1 \leq i \leq n$,
it follows that $x_j - \alpha \in \cS_j$. Define
$c \triangleq (c_1,c_2,\dots, c_n)$, where $c_i=x_i$ for $i \neq j$ and $c_j=x_j - \alpha$. This implies that
$$
\sum_{i=1}^n c_i = \sum_{i=1}^n x_i - \alpha = s - \alpha =0
$$
and hence $c \in \cC$ and $d(x,c) = 1$.

Assume now that there exists two distinct codewords $c,c' \in \cC$ such that $d(x,c) \leq 1$ and $d(x,c') \leq 1$.
By the triangle equality we have that
$$
d(c,c') \leq d(c,x) + d(x,c') \leq 2,
$$
a contradiction to the minimum distance of $\cC$.

Thus, $c$ is the unique codeword in $\cC$ such that ${d(x,c)=1}$.
\end{proof}

Theorem~\ref{thm:mixed_partitions} is generalized as follows.

\begin{theorem}
\label{thm:general_mixed}
Let $\{ \cG^*_1, \cG^*_2,\dots,\cG^*_n \}$ be a partition of $\cG \setminus \{ \bf0 \}$, where $\cG$ is an abelian group,
into $n$ nonempty subsets such that $\cG^*_i = \cG_i \setminus \{ \bf0 \}$ and $\cG_i$ is a subgroup of~$\cG$ whose size is $k_i$.
The code defined by
$$
\cC \triangleq \left\{ (c_1,c_2,\dots,c_n) ~:~ c_i \in \cG_i ,~ \sum_{i=1}^n c_i =0  \right\},
$$
where the sum is performed in $\cG$,
is a 1-perfect mixed code over ${\cG_1 \times \cG_2 \times \dots \times \cG_n}$.
\end{theorem}

Now, we give two examples of codes constructed by Theorem~\ref{thm:mixed_partitions}. We concentrate in cases where all the coordinates
are over the binary alphabet, except for one or two.

\begin{example}
\label{exm:Parts_STS}
Consider the space $\F_{2^n}$ and a subspace $\F_{2^k}$ of $\F_{2^n}$, $2 \leq k \leq n-1$.
The partition we form is to the nonzero elements of $\F_{2^k}$ and the rest, $2^n - 2^k$ elements,
of the nonzero elements of~$\F_{2^n}$ each one as a singleton.
Using Theorem~\ref{thm:mixed_partitions} we form a 1-perfect mixed code over $\F_2^{2^n -2^k} \times \F_{2^k}$. The codewords of weight three in this
code are a mixed Steiner system \textup{MS}$(2,3,\Z_2^{2^n - 2^k} \times \Z_{2^k})$.
\end{example}

\begin{example}
Consider the space $\F_{2^n}$ and two arbitrary disjoint subspaces $\F_{2^k}$ and $\F_{2^r}$ of $\F_{2^n}$,
where $2 \leq k \leq r \leq n-2$ and $k+r \leq n$.
The partition we form is to the nonzero elements of $\F_2^k$, the nonzero elements of $\F_{2^r}$,
and the rest, $2^n - 2^k -2^r +1$ elements, of the nonzero elements of $\F_2^n$ each one as a singleton.
Using Theorem~\ref{thm:mixed_partitions} we form a 1-perfect mixed code over $\F_2^{2^n-2^k-2^r+1} \times \F_{2^r} \times \F_{2^k}$.
The codewords of weight three in this code are a mixed Steiner system \textup{MS}$(2,3,\Z_2^{2^n-2^k-2^r+1} \times \Z_{2^r} \times \Z_{2^k})$.
\end{example}

\section{Large Sets and Resolutions}
\label{sec:largesets}

A \emph{large set} of Steiner systems S$(t,k,n)$, on an $n$-set $Q$, is a partition of all $k$-subsets
of $Q$ into Steiner systems S$(t,k,n)$. If we restrict ourself to Steiner systems S$(t-1,t,n)$, then exactly
two families of large sets are completely solved. A Steiner system S$(1,2,n)$ exists if and only if $n$ is even and its
large set is known as a one-factorization of the complete graph $K_n$. The existence of such one-factorizations
is a folklore and a survey can be found in~\cite{Wal97}. A Steiner system S$(2,3,n)$ is known as a Steiner triple
system and the corresponding large set is known to exist for every $n \equiv 1$ or $3 ~(\text{mod}~6)$, where $n > 7$.
It was first proved by Lu~\cite{Lu83,Lu84}, who left six open cases which were solved by Teirlinck~\cite{Tei91}.
An alternative shorter proof was given later by Ji~\cite{Ji05}. Unfortunately, there is no known construction for large sets of Steiner quadruple
systems and the one that gets close to it was constructed in~\cite{EtHa91}.
Large sets is a special case of resolvable designs~\cite{HRW72}. A Steiner system S$(t+1,k,n)$ is \emph{$t$-resolvable} if the blocks of the system can be
partitioned into subsets, each one is a Steiner system S$(t,k,n)$. The most remarkable result on resolvable Steiner systems
and large set was done by Keevash~\cite{Kee18}, but the proofs for their existence is probabilistic and no resolvable system
or a large set was constructed. The construction from large sets and resolutions will
be described in Section~\ref{sec:constructLS}. In Section~\ref{sec:nonexist} we prove that there are no large sets of mixed Steiner systems.

\subsection{Construction from Resolutions Large Sets}
\label{sec:constructLS}

\begin{construction}
\label{const:fromLS}
Let $\cS=(\Z_n,B)$ be a Steiner system \textup{S}$(t,k,n)$.
Let $\{ \cT_1, \cT_2,\ldots , \cT_r \}$ be a partition of $\cS$
into $r$ Steiner systems, where each $\cT_i$ is a Steiner system \textup{S}$(t-1,k,n)$.
Let $\cS'$ be the system on $\Z_2^n \times \Z_{r+1}$ whose blocks are
$$
\{ \{x_1,x_2,\ldots,x_k,(n+1,j)\} \}, ~~ \text{where} ~~ \{x_1,x_2,\ldots,x_k\} \in \cT_j ~.
$$
\end{construction}

\begin{theorem}
\label{thm:fromLS}
The system $\cS'$ generated by Construction~\ref{const:fromLS} is an \textup{MS}$(t,k+1,\Z_2^n \times \Z_{r+1})$.
\end{theorem}
\begin{proof}
By the definition all the blocks of $\cS'$ are of size $k+1$. There are two types of blocks of size~$t$.
For a $t$-subset $\{ x_1,x_2,\ldots,x_t \} \subset \Z_n$ is contained exactly once in $\cS$ and hence it is contained
in exactly one block of $\cS'$. Now, consider the $t$-subset $\{ x_1,x_2,\ldots,x_{t-1}, (n+1,j) \} \subset \Z_2^n \times \Z_{r+1}$.
Since $\cS_j$ is a Steiner system S$(t-1,k,n)$ it contained $\{ x_1,x_2,\ldots,x_{t-1} \}$ exactly once and hence there
is exactly one block in $\cS'$ that contains $\{ x_1,x_2,\ldots,x_{t-1} , (n+1,j) \}$.

As for the minimum distance, consider the two blocks $x=\{ x_1,x_2,\ldots,x_k, (n+1,i) \}$ and $y=\{ y_1,y_2,\ldots,y_k, (n+1,j) \}$
of $\Z_2^n \times \Z_{r+1}$. If $i=j$ then $\{ x_1,x_2,\ldots,x_k \}$ and $\{ y_1,y_2,\ldots,y_k \}$ are two blocks
in $\cS_i$, a Steiner system S$(t-1,k,n)$ and hence $d(x,y)=2(k-(t-1))+2=2k-2t+4$.
If $i \neq j$ then $\{ x_1,x_2,\ldots,x_k \}$ and $\{ y_1,y_2,\ldots,y_k \}$ are two blocks
in $\cS$, a Steiner system S$(t,k,n)$ and hence $d(x,y)=2(k-t)+2+1=2k-2t+3$.
Thus, the minimum Hamming distance of $\cS'$ is $2(k-t)+3$ as required.
\end{proof}

Theorem~\ref{thm:fromLS} can be applied for various parameters and five examples follow. All the constructed mixed Steiner systems
in these examples attain the lower bound of~\ref{lem:lower_bound_length}.

\begin{example}
\label{exm:LS_STS}
A Steiner system \textup{S}$(1,2,n)$ exists for each even $n$ and its is equivalent to a one-factor of $K_n$ and a large set
is a one-factorization of $K_n$~\cite{Wal97}. Hence, by Construction~\ref{const:fromLS}
and Theorem~\ref{thm:fromLS} there exists a mixed Steiner system \textup{MS}$(2,3,\Z_2^n \times Z_n)$ for any even integer $n$.
Each such mixed Steiner system attains the lower bound of Lemma~\ref{lem:lower_bound_length}.
\end{example}

\begin{example}
A large set of Steiner system S$(2,3,n)$ exists if and only if $n \equiv 1$ or $3 ~(\mmod ~6)$, $n>7$.
Such a large set contains $n-2$ Steiner systems and hence by Construction~\ref{const:fromLS}
and Theorem~\ref{thm:fromLS} there exists a mixed Steiner system \textup{MS}$(3,4,\Z_2^n \times Z_{n-1})$ for
$n \equiv 1$ or $3 ~(\mmod ~6)$, $n >7$.
Each such mixed Steiner system attains the lower bound of Lemma~\ref{lem:lower_bound_length}.
\end{example}

\begin{example}
A Steiner system \textup{S}$(2,4,n)$ exists if and only if $n \equiv 1$ or $4 ~(\mmod ~12)$.
It was proved in~\cite{Ji12} that 2-resolvable
Steiner systems \textup{S}$(3,4,n)$ exists if and only if $n \equiv 4 ~(\mmod ~12)$.
Such a resolution contains $(n-2)/2$ Steiner systems \textup{S}$(2,4,n)$ and hence by Construction~\ref{const:fromLS}
and Theorem~\ref{thm:fromLS} there exists a mixed Steiner system \textup{MS}$(3,5,\Z_2^n \times Z_{n/2})$ for
$n \equiv 4 ~(\mmod ~12)$. The mixed Steiner of Theorem~\ref{thm:mixedPreparata} is a system with such parameters.
Each such mixed Steiner system attains the lower bound of Lemma~\ref{lem:lower_bound_length}.
\end{example}

\begin{example}
A Steiner system \textup{S}$(2,3,n)$ exists if and only if $n \equiv 1$ or $3 ~(\mmod ~6)$.
A 1-resolvable \textup{S}$(2,3,n)$ into Steiner systems \textup{S}$(1,3,n)$ if and only if $n \equiv 3~(\mmod~6)$~\cite{RaWi71}.
Such a resolution contains $(n-1)/2$ Steiner systems \textup{S}$(1,3,n)$ and hence by Construction~\ref{const:fromLS}
and Theorem~\ref{thm:fromLS} there exists a mixed Steiner system \textup{MS}$(2,4,\Z_2^n \times Z_{(n+1)/2})$ for
$n \equiv 3 ~(\mmod ~6)$.
Each such mixed Steiner system attains the lower bound of Lemma~\ref{lem:lower_bound_length}.
\end{example}

\begin{example}
A Steiner system \textup{S}$(2,4,n)$ exists if and only if $n \equiv 1$ or $4 ~(\mmod ~12)$.
A 1-resolvable \textup{S}$(2,4,n)$ into Steiner systems \textup{S}$(1,4,n)$ if and only if $n \equiv 4~(\mmod~12)$.
Such a resolution contains $(n-1)/3$ Steiner systems \textup{S}$(1,4,n)$ and hence by Construction~\ref{const:fromLS}
and Theorem~\ref{thm:fromLS} there exists a mixed Steiner system \textup{MS}$(2,5,\Z_2^n \times Z_{(n+2)/3})$ for
$n \equiv 4 ~(\mmod ~12)$.
Each such mixed Steiner system attains the lower bound of Lemma~\ref{lem:lower_bound_length}.
\end{example}

\subsection{The Nonexistence of Large Sets}
\label{sec:nonexist}

Construction~\ref{const:fromLS} can be used with large sets of mixed Steiner systems instead of Steiner systems.
Unfortunately, it cannot help to obtain new mixed Steiner systems since there are no large sets of mixed Steiner systems
which are not Steiner systems.
Simple arguments leads to a generalization of Lemma~\ref{lem:derived} to large sets.
\begin{lemma}
\label{lem:shortLS}
If there exists a large set of \textup{MS}$(t,k,Q \times \Z_q)$, then there exists a large set of \textup{MS}$(t-1,k-1,Q)$
\end{lemma}

\begin{theorem}
There is no large set of \textup{MS}$(t,k,Q)$, unless $Q$ all the coordinates of $Q$ are over alphabets of the same size.
\end{theorem}
\begin{proof}
By using $t-1$ iterations of Lemma~\ref{lem:shortLS} the existence of a large set of MS$(t,k,Q)$ implies the
existence of a large set of MS$(1,k-t+1,Q')$, where $Q'$ is obtained by removing any $t-1$ coordinates from $Q$.
Therefore, it is sufficient to prove that there is no large set of MS$(1,k,Q)$. W.l.o.g. we can assume that
$Q = \Z_{q_1} \times \Z_{q_2} \times \cdots \times \Z_{q_{n-1}} \times \Z_{q_n}$, where $q_i \leq q_j$ for $i < j$ and not all the $q_i$'s
are equal.

In MS$(1,k,Q)$ the number of codewords that contain a nonzero element from $\Z_{q_n}$ is $q_n -1$ and they cover
exactly $(q_n -1)(k-1)$ nonzero elements of $Q$. The total number of words of weight $k$, that contain a nonzero element of $\Z_{q_n}$, in $Q$, is
$$
(q_n -1) \sum_{\cR \subset Q} \prod_{j=1}^{k-1} (q_{i_j} -1) ,
$$
where $\cR = \{x_1,x_2,\ldots , x_{k-1}\}$, $x_j \in \Z_{q_{i_j}} \setminus \{ \bf0 \}$ for
$1 \leq i_1 < i_2 < \ldots < i_{k-1} \leq n-1$. There are $q_n -1$ nonzero elements in $\Z_{q_n}$ and
therefore the size of a large set should be
\begin{equation}
\label{eq:Nexist1}
\frac{(q_n -1) \sum_{\cR \subset Q} \prod_{j=1}^{k-1} (q_{i_j} -1)}{q_n -1} = \sum_{\cR \subset Q} \prod_{j=1}^{k-1} (q_{i_j} -1) ~.
\end{equation}

The number of nonzero elements in $Q$ that have to be covered with codewords that do not contain a nonzero element of $\Z_{q_n}$ is
$\sum_{i=1}^{n-1} (q_i -1) - (q_n -1)(k-1)$. Therefore, the number of codewords (whose size is $k$) that do not contain nonzero elements of $\Z_{q_n}$ is
$$
\frac{\sum_{i=1}^{n-1} (q_i -1) - (q_n -1)(k-1)}{k}~.
$$
The total number of words of weight $k$, that do not contain a nonzero element of $\Z_{q_n}$, in $Q$, is
$$
\sum_{\cR \subset Q} \prod_{j=1}^k (q_{i_j} -1),
$$
where $\cR = \{x_1,x_2,\ldots , x_k\}$, $x_j \in \Z_{q_{i_j}} \setminus \{ \bf0 \}$ for
$1 \leq i_1 < i_2 < \ldots < i_k \leq n-1$.
This implies that the large set has size
\begin{equation}
\label{eq:Nexist2}
\frac{k \sum_{\cR \subset Q} \prod_{j=1}^k (q_{i_j} -1)}{\sum_{i=1}^{n-1} (q_i -1) - (q_n -1)(k-1)}.
\end{equation}

The two expressions of Equations~(\ref{eq:Nexist1}) and~~(\ref{eq:Nexist2}) that we got for the
size of a large set must be equal and hence
$$
\frac{k \sum_{\cR \subset Q} \prod_{j=1}^k (q_{i_j} -1)}{\sum_{i=1}^{n-1} (q_i -1) - (q_n -1)(k-1)} = \sum_{\cR \subset Q} \prod_{j=1}^{k-1} (q_{i_j} -1)~
$$
or
$$
k \sum_{\cR \subset Q} \prod_{j=1}^k (q_{i_j} -1) = \left( \sum_{\cR \subset Q} \prod_{j=1}^{k-1} (q_{i_j} -1) \right)
\left(  \sum_{i=1}^{n-1} (q_i -1) - (q_n -1)(k-1) \right).
$$
Careful analysis of the required equality yields that the left side of the equation is larger that the right side, unless all the $q_i$'s
are equal in which case there is indeed equality.
\end{proof}

\section{Pairwise Balanced Mixed Systems}
\label{sec:pairwise}

In this section we concentrate on constructions for mixed Steiner systems MS$(2,k,Q)$.
Example~\ref{exm:LS_STS} based on Construction~\ref{const:fromLS} and Theorem~\ref{const:fromLS} implies
the existence of an MS$(2,3,\Z_2^n \times \Z_n)$ for all even~$n$.
Example~\ref{exm:Parts_STS} based on 1-perfect mixed codes implies the existence of
an MS$(2,3,\Z_2^{2^n - 2^k} \times \Z_{2^k})$ for all $n \geq 3$ and $2 \leq k \leq n-1$.
Two different constructions for mixed Steiner systems MS$(2,k,Q)$ will be presented.
The first one presented in Section~\ref{sec:OA_pairs} is based on orthogonal arrays.
The second one presented in Section~\ref{sec:OF_pairs} is based on covering all pairs by triples and
pairwise disjoint one-factors of the complete graph $K_n$.
This construction will be used only for constructions of mixed Steiner systems MS$(2,3,\Z_2^n \times \Z_q)$, where $q <n$.

\subsection{A Construction Based on Orthogonal Arrays}
\label{sec:OA_pairs}

A \emph{Latin square of order $k$} is a $k \times k$ array in which each row and each column is a permutation of a given $k$-set.
Two $k \times k$ Latin squares $\cA$ and $\cB$ are \emph{orthogonal} if all the ordered pairs $( \cA(i,j),\cB(i,j) )$, $1 \leq i,j \leq k$,
are distinct.
An \emph{orthogonal array} OA$(2,n,k)$ is a $k^2 \times n$ array $\cA$ over $\Z_k$ in which in any projection
of two columns from $\cA$ each ordered pair of $\Z_k$ appears exactly once
The existence problem of orthogonal array OA$(2,n,k)$ was extensively studied in the literature~\cite{HSS99}.
Such an array is equivalent to a set of $n-2$ mutially orthogonal Latin squares of order $k$ (see Chapter 2 of~\cite{Etz22}).

\begin{construction}
\label{const:fromOA}
Let $\cA$ be an orthogonal array \textup{OA}$(2,n,k)$ and let $t$ be an integer, $1 \leq r \leq n-1$.
Let $\cS$ be the system on $\Z_2^{rk} \times \Z_{k+1}^{n-r}$ whose blocks are
$$
\{ ik+1,ik+2,\ldots, ik+k\}, ~~ 0 \leq i \leq r-1
$$
$$
\{ j_1, k+j_2,\ldots,(r-1)k + j_r , (tk+1,j_{r+1}), (rk+2,j_{r+2}),\ldots, (rk+n-r,j_n)  \}, ~~~ (j_1,j_2,\ldots,j_n) \in \cA ~.
$$
\end{construction}

\begin{theorem}
\label{thm:MSfromOA}
If $r=n-1$ then the system $\cS$ of Construction~\ref{const:fromOA} is an \textup{MS}$(2,k,\Z_2^{(n-1)k} \times \Z_{k+1})$.
\end{theorem}
\begin{proof}
It is easily verified that each 2-subset of $\Z_2^{(n-1)k} \times \Z_{k+1}$ is contained in exactly one codeword ($k$-subset).
Consider the two codewords $x=(x_1,x_2,\ldots,x_{(n-1)k},\alpha)$ and $y=(y_1,y_2,\ldots,y_{(n-1)k},\beta)$.
If $\alpha = \beta$, then $x_i =1$ implies that $y_i=0$ and $y_i =1$ implies that $x_i=0$. As a consequence $d(x,y) = 2(k-1) > 2(k-2)+1$.
If $\alpha \neq \beta$, then for at most one $i$, $1 \leq i \leq (n-1)k$ we have $x_i=y_i=1$ and hence
$d(x,y) \geq 2(k-2)+1$. Thus, $\cS$ is an \textup{MS}$(2,k,\Z_2^{(n-1)k} \times \Z_{k+1})$.
\end{proof}

\begin{corollary}
\label{cor:MSfromOA}
If there exists an orthogonal array \textup{OA}$(2,n,k)$ then there exists a mixed Steiner system
\textup{MS}$(2,k,\Z_2^{(n-1)k} \times \Z_{k+1})$.
\end{corollary}

\begin{theorem}
\label{thm:fromOA}
If $1 \leq r < n-1$, then
the system $\cS$ of Construction~\ref{const:fromOA} is an \textup{MS}$_d(2,k,\Z_2^{rk} \times \Z_{k+1}^{n-r})$,
where $d=n+r-2$.
\end{theorem}
\begin{proof}
It easily verified that each $2$-subset of $\Z_2^{rk} \times \Z_{k+1}^{n-r}$ is contained in exactly one codeword ($k$-subset).
Consider two codewords $x=(x_1,x_2)$ and $y=(y_1,y_2)$, where $x_1$ and $y_1$ have length $rk$ and $x_2$ and $y_2$ have length $n-r$.
If $d(x_2,y_2)=n-r-1$ then $d(x_1,y_1) =2r$ and hence $d(x,y)=n+r-1$.
If $d(x_2,y_2)=n-r$ then $d(x_1,y_1) \geq 2(r-1)$ and hence $d(x,y)=n+r-2$.
Thus, $\cS$ is an \textup{MS}$_d(2,k,\Z_2^{(n-1)k} \times \Z_{k+1})$, where $d = n+r-2$.
\end{proof}

\begin{corollary}
\label{cor:fromOA}
If there exists an \textup{OA}$(2,n,k)$ then there exists an \textup{MS}$_d(2,k,\Z_2^{rk} \times \Z_{k+1}^{n-r})$,
$d=n+r-2$, for each $r$, $1 \leq r < n-1$.
\end{corollary}


Finally, we present an example for a construction of mixed Steiner system MS$(2,3,\Z_2^{2k} \times \Z_{k+1})$.

\begin{example}
Let $\cA$ be an orthogonal array \textup{OA}$(2,3,k)$ (equivalent to a Latin square of order $k$ which exist for all $k > 1$).
Applying Construction~\ref{const:fromOA} we obtain by Theorem~\ref{thm:fromOA} a mixed Steiner system
\textup{MS}$(2,3,\Z_2^{2k} \times \Z_{k+1})$.
\end{example}

\subsection{A Construction Based on a Two Types of Subsets}
\label{sec:OF_pairs}

The next construction is a generalization of Construction~\ref{const:fromLS}.
The construction is based on a new type of an intriguing combinatorial design.
The design is based on a combination of $k$-subsets and $(k-1)$-subsets.
The construction has four parameters, the first two are $n$ and $q$, where the codewords of the system
are on $\Z_2^n \times \Z_q$. The other two parameters are $k$ and $t$.
There are two sets in the construction, a first set $\cR$ of $k$-subsets of $\Z_n$ in which each
$t$-subset of $\Z_n$ is contained in at most one $k$-subset of $\cR$.
The second set $\cT$, contains $r$ subsets, $\cT = \{\cT_1,\cT_2,\ldots,\cT_r\}$, each $\cT_i$ is a Steiner system S$(t-1,k-1,n)$ on $\Z_n$
and each two distinct $\cT_i$'s are disjoint. Moreover, each $t$-subset of $\Z_n$ which is not contained in a $k$-subset of $\cR$
is contained in exactly one of the $\cT_i$'s.

\begin{construction}
\label{const:recursive_k_k1}
Given the parameters $n$, $q$, $k$, and $t$, and the sets $\cR$ and $\cT$, we construct the following sets of $k$-subset
to form a code $\cC$.
$$
\{ \{x_1,x_2,\ldots x_k \} ~:~ \{ x_1,x_2,\ldots,x_k \} \in \cR \}
$$
and
$$
\{ \{ x_1,x_2,\ldots,x_{k-1},(n+1,i) \} ~:~ \{ x_1,x_2,\ldots,x_{k-1} \} \in \cT_i ,~~ 1 \leq i \leq r \} ~.
$$
\end{construction}

\begin{theorem}
\label{thm:recursive_k_k1}
The code $\cC$ formed in Construction~\ref{const:recursive_k_k1} is a mixed Steiner system \textup{MS}$(t,k,\Z_2^n \times \Z_{r+1})$.
\end{theorem}
\begin{proof}
By the definition of the set $\cR$ and the $\cT_i$'s the codewords have weight $k$ and each $t$-subset of $\Z_n$ is
contained in exactly one codeword of $\cC$.

For a word $\{ x_1,x_2,\ldots,x_{t-1},(n+1,i)\}$, where
$\{ x_1,x_2,\ldots,x_{t-1} \} \in \Z_n$ and $i \in \Z_q \setminus \{ \bf0 \}$. The word $\{x_1,x_2,\ldots,x_{t-1}\}$
is contained in on $(k-1)$-subset of $\cT_i$ since $\cT_i$ is a Steiner system S$(t-1,k-1,n)$ and hence
$\{x_1,x_2,\ldots,x_{t-1},(n+1,i)\}$ is contained in a codeword of $\cC$.
\end{proof}

\begin{definition}
\label{dfn:nkPT}
An $(n,r)$-\emph{pairs-triples design} is a set $\cT$ of $r$ pairwise disjoint $r< n-1$ one-factors of the complete graph $K_n$
on the vertices $\Z_n$ and a set $\cR$ of triples of $\Z_n$ such that each pair of $\Z_n$ is contained exactly once either
in one-factor of $\cT$ or in $\cR$.
\end{definition}

\begin{remark}
Definition~\ref{dfn:nkPT} excluded the case when $r=n-1$ which is associated with a one-factorization of $K_n$
and with the system constructed in Construction~\ref{const:fromLS} and the mixed Steiner systems \textup{MS}$(2,3,\Z_2^n \times \Z_n)$
as constructed in Example~\ref{exm:LS_STS}.
\end{remark}

\begin{construction}
\label{const:PR_TR}
Let $\cT = \{ \cT_1 , \cT_2, \ldots \cT_r \}$, where $r > 1$, and $\cR$ be an $(n,r)$-pairs-triples design.
We form the following sets of codewords on $\Z_2^n \times \Z_{r+1}$:
$$
\{ \{x,y,z\} ~:~ \{x,y,z\} \in \cR \}
$$
and
$$
\{ \{ x,y,(n+1,i) \} ~:~ \{x,y\} \in \cT_i, ~ 1 \leq i \leq r \} ~.
$$
\end{construction}

The following theorem is a special case of Theorem~\ref{thm:recursive_k_k1}.

\begin{theorem}
\label{thm:PR_TR}
The code formed by the codewords in Construction~\ref{const:PR_TR} yields a mixed Steiner system \textup{MS}$(2,3,\Z_2^n \times \Z_{r+1})$.
\end{theorem}

\begin{corollary}
If there exists an $(n,r)$-pairs-triples design, then there exists a mixed Steiner system \textup{MS}$(2,3,\Z_2^n \times \Z_{r+1})$.
\end{corollary}

Constructions of sets of $k$-subsets and the $(k-1)$-subsets described in this section have their own interests
and they motivate future research. We continue with such subsets with $k=3$ and $t=2$.
Constructions of an $(n,r)$-pairs-triples designs are important in constructions of
MS$(2,3,\Z_2^n \times \Z_{r+1})$ and the rest of this section will consider some of their constructions.
The constructions have similarity to other constructions of disjoint triple systems, e.g.,~\cite{Etz92a,Etz92b}.

Theorem~\ref{thm:PR_TR} derived from Construction~\ref{const:PR_TR}
and the necessary conditions for the existence of an MS$(2,3,\Z_2^n \times Z_q)$ given in Theorem~\ref{thm:necessry2,3}
yield the following theorem.

\begin{theorem}
\label{thm:nkPT}
An $(n,r)$-pairs-triples design exists only if $n$ is even and $r$ is odd, $1 \leq r \leq n-1$.
If $n \equiv 2$ or $4~(\mmod ~6)$, then
$r \equiv n-1 ~(\mmod ~6)$, and if $n$ divisible by 6, then $r$ is odd.
\end{theorem}

\begin{lemma}
\label{lem:simplePT}
If $n \equiv 0$ or $2 ~ (\mmod ~ 6)$ then there exists an $(n,1)$-pairs-triples design.
\end{lemma}
\begin{proof}
Let $\cS$ be a Steiner triple system S$(2,3,n+1)$ on $\Z_{n+1}$. It is easy to verify that the set of triples
$$
\cR=\{ \{ x,y,z \} ~:~ \{ x,y,z \} \in \cS, ~ n \notin \{x,y,z\} \}
$$
and the set of pairs
$$
\cT_1=\{ \{ x,y \} ~:~ \{ x,y,n \} \in \cS \}
$$
form an $(n,1)$-pairs-triples design.
\end{proof}


\begin{conjecture}
\label{conj:nkPT}
An $(n,r)$-pairs-triples design exists for each $n$ divisible by 6 and odd $r$, $1 \leq r < n-1$ and
for each $n \equiv 2$ or $4~(\mmod ~6)$ and $r \equiv n-1 ~(\mmod ~6)$, $1 \leq r < n-1$.
\end{conjecture}

Lemma~\ref{lem:simplePT} and the constructed $(n,1)$-pairs-triples designs cannot be used for constructions of
mixed Steiner system by Construction~\ref{const:PR_TR}. But, these $(n,1)$-pairs-triples sets can be used in the following
recursive construction for $(n,r)$-pairs-triples designs for $r>1$. This construction is a doubling construction.

\begin{construction}
\label{const:recursivePT}
Let $T=\{ T_0 ,T_1,\ldots, T_{r-1} \}$ be a set of pairwise disjoint one-factors and let $R$ be a set of triples
in an $(n,r)$-pairs-triples design on $\Z_n$. We construct a set of triples and sets of pairs on the points of $\Z_n \times \Z_2$.
The first set is of triples
$$
\cR \triangleq \{ \{ (x,i),(y,i),(z,i) \} ~:~ \{ x,y,z \} \in R, ~~ i \in \{0,1\} \}~.
$$
The first $n$ sets of pairs are
$$
\cT_i \triangleq \{ \{ (x,0),(x+i,1) \} ~:~ 0 \leq x \leq n-1 \}, ~~  0 \leq i \leq n-1
$$
The next $r$ sets of pairs are
$$
\cT_{n+i} \triangleq \{ \{ (x,j),(y,j) \} ~:~ \{ x,y \} \in T_i , ~~ j \in \{ 0,1 \} \}, ~~ 0 \leq i \leq r-1 ~.
$$
\end{construction}

\begin{theorem}
\label{thm:recursivePT}
If there exists an $(n,r)$-pairs-triples design, then there exist a $(2n,n+r)$-pairs-triples design.
\end{theorem}
\begin{proof}
It is easy to verify that $\cR$ is a set of triples and each $\cT_i$, $0 \leq i \leq n-1$ is a one-factorization of $\Z_n \times \Z_2$.
All the pairs in these one-factorizations contain all pairs that are not contained in $\Z_n \times \{ j \}$, $j=0,1$.
The other $\cT_i$'s, $n \leq i \leq n+r-1$, are one-factorizations of $\Z_n \times \Z_2$ and all the pairs are contained
only in $\Z_n \times \{ j \}$, $j=0,1$. These pairs are exactly those pairs of $\Z_n \times \{ j \}$, $j=0,1$,
that are not contained in the triples of $\cR$ since the pairs of the $T_i$'s are not contained in the pairs of $R$.
\end{proof}

\begin{remark}
\label{rmk:doubling}
When Construction~\ref{const:recursivePT} is applied on $n \equiv 0~(\mmod~6)$ its outcome is for $2n \equiv 0~(\mmod~12)$.
When Construction~\ref{const:recursivePT} is applied on $n \equiv 2~(\mmod~6)$ its outcome is for $2n \equiv 4~(\mmod~12)$.
When Construction~\ref{const:recursivePT} is applied on $n \equiv 4~(\mmod~6)$ its outcome is for $2n \equiv 8~(\mmod~12)$.
Therefore, we do not have any constructions for $(n,r)$-pairs-triples designs when $n \equiv 2,~6$ or $10 ~(\mmod ~ 12)$.
Such construction can be obtained from some of the constructions in the sequel.
\end{remark}

In view of Remark~\ref{rmk:doubling} it is desired to find more values of $n$ for which $(n,r)$-pairs-triples designs exist.
Moreover, not all values of $r$ are covered in this way.

Similar to the doubling Construction~\ref{const:recursivePT} there is
a way we can construct a tripling construction.

\begin{construction}
\label{const:tripling}
Let $T=\{ T_0 ,T_1,\ldots, T_{r-1} \}$ be a set of pairwise disjoint one-factors and let $R$ be a set of triples
in an $(n,r)$-pairs-triples design on $\Z_n$. We construct a set of triples and
sets of pairs on the points of $\Z_n \times \Z_3$. The set of triples $\cR = \cR_1 \cup \cR_2$ is defined by
$$
\cR_1 \triangleq \{ \{ (x,i),(y,i),(z,i) \} ~:~ \{ x,y,z \} \in R, ~~ i \in \{0,1,2\} \}~.
$$
and
$$
\cR_2 \triangleq \{ (x,0),(y,1),(x+y,2) ~:~ x,y \in \Z_n \}~.
$$
The $r$ sets of pairs are defined by
$$
\cT_i \triangleq \{ \{ (x,j),(y,j) \} ~:~ \{ x,y \} \in T_i , ~~ j \in \{ 0,1,2 \} \}, ~~ 0 \leq i \leq r-1 ~.
$$
\end{construction}

Construction~\ref{const:tripling} yields the following theorem which have a similar proof to the one of Theorem~\ref{thm:recursivePT}.
Construction~\ref{const:tripling} is an example of a construction for $(n,r)$-pairs-triples designs for which $n \equiv 6~(\mmod~12)$.

\begin{theorem}
\label{thm:tripling}
If there exists an $(n,r)$-pairs-triples design, then there exist an $(3n,r)$-pairs-triples design.
\end{theorem}

The tripling construction of Construction~\ref{const:tripling} is generalized using any Steiner triple system.

\begin{construction}
\label{const:multiplySTS}
Let $T=\{ T_0 ,T_1,\ldots, T_{r-1} \}$ be a set of pairwise disjoint one-factors and let $R$ be a set of triples
in an $(n,r)$-pairs-triples design on $\Z_n$. Let $\cS$ be a Steiner triple system \textup{S}$(2,3,m)$.
We construct a set of triples and sets of pairs on the points of $\Z_n \times \Z_m$.
The set of triples $\cR = \cR_1 \cup \cR_2$ is defined by
$$
\cR_1 \triangleq \{ \{ (x,i),(y,i),(z,i) \} ~:~ \{ x,y,z \} \in R, ~~ i \in \Z_m \}~.
$$
and
$$
\cR_2 \triangleq \{ (x,i),(y,j),(x+y,\ell) ~:~ x,y \in \Z_n, ~~ \{i,j,\ell\} \in \cS \}~.
$$
The $r$ sets of pairs are defined by
$$
\cT_i \triangleq \{ \{ (x,j),(y,j) \} ~:~ \{ x,y \} \in T_i , ~~ j \in \Z_m \}, ~~ 0 \leq i \leq r-1 ~.
$$
\end{construction}

Theorem~\ref{thm:tripling} can be generalized using Construction~\ref{const:multiplySTS} to obtain the following result.

\begin{theorem}
\label{thm:Recur_STS}
If there exists an $(n,r)$-pairs-triples design, then there exists an $(mn,r)$-pairs-triples design, for each $m \equiv 1$ or $3~(\mmod~6)$.
\end{theorem}
\begin{proof}
The triples and the disjoint one-factorizations of Construction~\ref{const:tripling} are constructed on each triple of $\cS$.
Since $\cS$ is a Steiner triple system S$(2,3,m)$ this completes the proof.
\end{proof}

Construction~\ref{const:multiplySTS} can be improved when $m \equiv 3~(\mmod ~ 6)$ and given a specific construction with
a resolvable S$(2,3,m)$, we can replace $n$ parallel classes (S$(1,3,m)$ with $2n$ one-factors. This will be explained in details
in the next version of this draft.

Another generalization of Construction~\ref{const:recursivePT} yields the following theorem.

\begin{theorem}
\label{thm:recursivePTprime}
If there exists an $(n,r)$-pairs-triples design, then there exists an $(pn,(p-1)n+r)$-pairs-triples design, where $p$ is a prime number.
\end{theorem}

The proof of Theorem~\ref{thm:recursivePTprime} will be given in the next version of this paper.
A special case of Theorem~\ref{thm:recursivePTprime} will be proved now.

\begin{theorem}
\label{thm:recursivetripling}
If there exists an $(n,r)$-pairs-triples design, then there exists an $(3n,2n+r)$-pairs-triples design.
\end{theorem}
\begin{proof}
Let $T=\{ T_0 ,T_1,\ldots, T_{r-1} \}$ be a set of pairwise disjoint one-factors and let $R$ be a set of triples
in an $(n,r)$-pairs-triples design on $\Z_n$.
We construct a set of triples and sets of pairs on the points of $\Z_n \times \Z_3$.
The set of triples is defined by
$$
\cR_1 \triangleq \{ \{ (x,i),(y,i),(z,i) \} ~:~ \{ x,y,z \} \in R, ~~ i \in \Z_3 \}~.
$$
The first $\frac{3n}{2}$ sets of pairs are defined by
$$
\cT_{\frac{mn}{2}+i} \triangleq \Big{\{} \{(x,m),(x+i,m+1) \} : 0 \leq x \leq \frac{n}{2}-1 \Big{\}}
$$
$$
\bigcup \Big{\{} \{(x+\frac{n}{2},m+1),(x+i,m+2) \} ~:~ 0 \leq x \leq \frac{n}{2}-1 \Big{\}}
$$
$$
\bigcup  \Big{\{} \{(x+\frac{n}{2},m),(x+\frac{n}{2}+i,m+2) \Big{\}} ~:~ 0 \leq x \leq \frac{n}{2}-1 \Big{\}} ,~~ 0 \leq i \leq \frac{n}{2} -1, ~~ m \in \Z_3~.
$$
The next $\frac{n}{2}$ sets of pairs are defined by
$$
\cT_{\frac{3n}{2}+i} \triangleq \Big{\{} \{(x,m),(x+\frac{n}{2}+i,m+1) \} : 0 \leq x \leq \frac{n}{2}-1 \Big{\}} ,~~ 0 \leq i \leq \frac{n}{2} -1 ~.
$$
The last $r$ sets of pairs are defined by
$$
\cT_{2n+i} \triangleq \{ \{ (x,j),(y,j) \} ~:~ \{ x,y \} \in T_i , ~~ j \in \Z_3 \}, ~~ 0 \leq i \leq r-1 ~.
$$
\end{proof}

Another type of recursive construction based on two pairs-triples designs is as follows.
\begin{construction}
\label{const:recursiveMN}
Let $T^i=\{ T^i_0 ,T^i_1,\ldots, T^i_{r-1} \}$, $i=1,2$, be a set of pairwise disjoint one-factors and $R^i$, $i=1,2$, be a set of triples
in an $(n_i,r_i)$-pairs-triples design, $i=1,2$, on $\Z_n$.
We construct a set of triples and sets of pairs on the points of $\Z_{n_1} \times \Z_{n_2}$.
The set of triples $\cR = \cR_1 \cup \cR_2$ is defined by
$$
\cR_1 \triangleq \{ \{ (x,i),(y,i),(z,i) \} ~:~ \{ x,y,z \} \in R^1, ~~ i \in \Z_{n_2} \}~.
$$
and
$$
\cR_2 \triangleq \{ (x,i),(y,j),(x+y,i+j) ~:~ x,y\in \Z_{n_1}, ~~ \{ i,j,\ell \} \in R^2 \}~.
$$
The first $r_1$ sets of pairs are defined by
$$
\cT_i \triangleq \{ \{ (x,j),(y,j) \} ~:~ \{ x,y \} \in T^1_i , ~~ j \in \Z_{n_2} \}, ~~ 0 \leq i \leq r_1-1 ~.
$$
The next $n_1 r_2$ sets of pairs are defined by
$$
\cT_{r_1 +n_1 \ell +j} \triangleq \{ \{ (x,i),(x+y,j) \} ~:~  x \in \Z_{n_1} , ~ \{ i,j \} \in T^2_\ell \},
y \in \Z_{n_1}, ~ 0 \leq \ell \leq r_2-1 ~.
$$
\end{construction}

Construction~\ref{const:recursiveMN} is used to obtain the following theorem proved similarly to the proofs of the other
theorem associated with the previous recursive constructions.

\begin{theorem}
\label{thm:recursiveMN}
If there exists an $(n_1,r_1)$-pairs-triples design and an $(n_2,r_2)$-pairs-triples design, then there exists
an $(n_1 n_2,r_1 + n_1 r_2)$-pairs-triples design.
\end{theorem}

To apply all the recursive constructions we will need more initial condition. We will consider now all the $(n,r)$-pairs-triples designs
for $n=6,~8,~10$, and~12.

For $n=6$, a $(6,1)$-pairs-triples design is covered by Lemma~\ref{lem:simplePT}.

\begin{example}
A $(6,3)$-pairs-triples design on the points of $\Z_6$ is given by the set of triples
$$
\cR \triangleq \{ \{0,1,2\},\{3,4,5\}  \}
$$
and the 3 sets of pairs
$$
\cT_0 \triangleq \{ \{ 0,3 \}, \{1,4 \}, \{ 2,5 \} \},
$$
$$
\cT_1 \triangleq \{ \{ 0,4 \}, \{1,5 \}, \{ 2,3 \} \},
$$
and
$$
\cT_2 \triangleq \{ \{ 0,5 \}, \{1,3 \}, \{ 2,4 \} \}.
$$
\end{example}

For $n=8$, $(8,1)$-pairs-triples design is covered by Lemma~\ref{lem:simplePT}. There are no other $(8,r)$-pairs-triples designs since
by Theorem~\ref{thm:nkPT} we have that $r \equiv 1~(\mmod~6)$, $1 \leq k < 7$.

For $n=10$, a $(10,1)$-pairs-triples design does not exist by Theorem~\ref{thm:nkPT}.
By Theorem~\ref{thm:nkPT} we have that for a $(10,r)$-pairs-triples designs we have
that $r \equiv 3~(\mmod~6)$, $1 \leq r < 9$ and hence we can have
$(10,3)$-pairs-triples design. This set is covered in the following example.

\begin{example}
A $(10,3)$-pairs-triples design on the points of $\Z_{10}$ is given by the set of triples
$$
\cR \triangleq \{ \{i,i+1,i+3\} ~:~ 0 \leq i \leq 9  \}
$$
and the 3 sets of pairs
$$
\cT_0 \triangleq \{ \{ 0,5 \}, \{ 1,7 \}, \{ 2,6 \}, \{ 3,9 \}, \{ 4,8 \} \},
$$
$$
\cT_1 \triangleq \{ \{ 0,6 \}, \{ 1,5 \}, \{ 2,8 \}, \{ 3,7 \}, \{ 4,9 \} \},
$$
and
$$
\cT_2 \triangleq \{ \{ 0,4 \}, \{ 1,6 \}, \{ 2,7 \}, \{ 3,8 \}, \{ 5,9 \} \}.
$$
\end{example}

For $n=12$, $(12,1)$-pairs-triples design is covered by Lemma~\ref{lem:simplePT}.
By Theorem~\ref{thm:nkPT} we have that for $(12,r)$-pairs-triples designs $r$ is odd, $1 \leq r < 11$.
$(12,7)$-pairs-triples design and $(12,9)$-pairs-triples design are obtained from
$(6,1)$-pairs-triples design and $(6,3)$-pairs-triples design, respectively, using Construction~\ref{const:recursivePT}.
A $(12,3)$-pairs-triples design is covered in the following example.

\begin{example}
A $(12,3)$-pairs-triples design on the points of $\Z_{12}$ is given by the set of triples
$$
\cR \triangleq \{ \{i,i+2,i+5\} ~:~ 0 \leq i \leq 11  \} \cup \{ \{i,i+4,i+8\} ~:~ 0 \leq i \leq 3  \}
$$
and the 3 sets of pairs
$$
\cT_0 \triangleq \{ \{ i,i+6 \} ~:~ 0 \leq i \leq 5  \},
$$
$$
\cT_1 \triangleq \{ \{ 2i,2i+1 \} ~:~ 0 \leq i \leq 5  \},
$$
and
$$
\cT_2 \triangleq \{ \{ 2i+1,2i+2 \} ~:~ 0 \leq i \leq 5  \}.
$$
\end{example}

\begin{example}
A $(12,5)$-pairs-triples design on the points of $\Z_{12}$ is given by the set of triples
$$
\cR \triangleq \{ \{i,i+2,i+5\} ~:~ 0 \leq i \leq 11  \}
$$
and the five sets of pairs
$$
\cT_0 \triangleq \{ \{ i,i+6 \} ~:~ 0 \leq i \leq 5  \},
$$
$$
\cT_1 \triangleq \{ \{ 2i,2i+1 \} ~:~ 0 \leq i \leq 5  \},
$$
$$
\cT_2 \triangleq \{ \{ 0,4 \}, \{ 5,9 \}, \{ 6,10 \}, \{ 3,11 \}, \{ 1,2 \}, \{ 7,8 \} \},
$$
$$
\cT_3 \triangleq \{ \{ 0,8 \}, \{ 1,5 \}, \{ 2,6 \}, \{ 7,11 \}, \{ 3,4 \}, \{ 9,10 \} \},
$$
and
$$
\cT_4 \triangleq \{ \{ 1,9 \}, \{ 2,10 \}, \{ 3,7 \}, \{ 4,8 \}, \{ 0,11 \}, \{ 5,6 \} \}.
$$
\end{example}

\section{Conclusion and Future Research}
\label{sec:conclusion}

A comprehensive discussion on the properties, bounds, constructions, and nonexistence results of mixed Steiner
system was given. The draft will be updated with further constructions,
proofs that were omitted, and table of parameters, will be added in the next version.
The direction of research presented in this paper poses many open problems. For example, proving Conjecture~\ref{conj:nkPT},
or finding more mixed Steiner systems MS$(t,k,Q)$, especially when $t > 2$.

The pairs-triples designs were extensively studied, but what about a triples-quadruples design which are
required in Construction~\ref{const:recursive_k_k1} for $k=4$ and $t=3$. Similar questions can be asked about other values of $k$ and $t$.


%

\end{document}